\documentclass[11pt,a4paper]{article}

\usepackage{amsmath}
\usepackage{amsfonts,amssymb}
\usepackage{enumerate}
\usepackage{mathrsfs}
\usepackage{amsthm}
\usepackage{a4wide}

\theoremstyle{plain}
\newtheorem{theo}{Theorem}[section]
\newtheorem{prop}[theo]{Proposition}
\newtheorem{lemm}[theo]{Lemma}

\theoremstyle{definition}
\newtheorem{rema}[theo]{Remark}

\DeclareMathOperator{\supp}{supp}

\def\Deltayx{\Delta_{x,y}}

\def\eps{\varepsilon}

\def\la{\left\lvert}
\def\lA{\left\lVert}
\def\le{\leq}

\def\mez{\frac{1}{2}}

\def\partialx{\nabla}
\def\partialyx{\nabla_{x,y}}

\def\ra{\right\rvert}
\def\rA{\right\rVert}

\def\xR{\mathbf{R}}

\begin{document}

\title{The water waves equations: from Zakharov to Euler}

\author{
T. Alazard, 
N. Burq, 
C. Zuily}
\date{\empty}

\maketitle

\abstract{Starting form the Zakharov/Craig-Sulem formulation of the gravity water waves equations, 
we prove that one can define a pressure term and hence obtain a solution 
of the classical Euler equations. 
It is proved that these results hold  in rough domains, under minimal assumptions 
on the regularity 
to ensure, in terms of Sobolev spaces, that the solutions are $C^1$.}

\section{Introduction}

We study the dynamics of an incompressible layer of inviscid liquid, 
having constant density, occupying 
a fluid domain with a free surface.


We begin by describing the fluid domain. Hereafter, $d\ge 1$, $t$ 
denotes the time variable and $x\in \xR^d$ and $y\in \xR$ denote the horizontal and vertical variables. 
We work in a fluid domain with free boundary  of the form
$$
\Omega=\{\,(t,x,y)\in (0,T)\times\xR^d\times\xR \, : \, (x,y) \in \Omega(t)\,\},
$$
where $\Omega(t)$ is the $d+1$-dimensional 
domain located between two hypersurfaces: 
a free surface denoted by $\Sigma(t)$ which 
will be supposed to be a graph and a fixed bottom $\Gamma$. 
For each time $t$, one has
$$
\Omega(t)=\left\{ (x,y)\in \mathcal{O} \, :\, y < \eta(t,x)\right\},
$$
where $\mathcal{O}$ is a given open connected domain and where $\eta$ is the free surface elevation. 
We denote by $\Sigma$ the free surface:
$$
\Sigma = \{(t,x,y): t\in(0,T), (x,y)\in \Sigma(t)\},  
$$
where $  \Sigma(t)=\{ (x,y)\in\xR^d\times \xR\,:\, y=\eta(t,x)\}$ 
and we set $\Gamma=\partial\Omega(t)\setminus \Sigma(t)$.

Notice that $\Gamma$ does not depend on time. 
Two classical examples are the case of infinite depth 
($\mathcal{O}=\xR^{d+1}$ so that $\Gamma=\emptyset$) 
and the case where the bottom is the graph of a function (this corresponds to the case 
$\mathcal{O}=\{(x,y)\in\xR^d\times \xR\,:\, y> b(x)\}$ for some given function $b$).

We introduce now a condition which ensures  that, at time $t$, there exists a fixed strip   separating  the free surface from the bottom.     
\begin{equation}\label{eta*}
(H_t):\qquad  \exists  h>0 : \quad \Gamma 
\subset \{(x,y)\in \xR^d \times \xR: y<\eta(t, x)-h\}.
\end{equation}
No regularity assumption will be made on the bottom $\Gamma$.

\subsubsection*{The incompressible Euler equation with free surface}


%
Hereafter, we use the following notations
$$
\partialx=(\partial_{x_i})_{1\le i\le d},\quad \partialyx =(\partialx,\partial_y), 
\quad \Delta = \sum_{1\le i\le d} \partial_{x_i}^2,\quad 
\Deltayx = \Delta+\partial_y^2.
$$
The Eulerian velocity field $v\colon \Omega \rightarrow \xR^{d+1}$ 
solves the incompressible Euler equation
\begin{equation*}
\partial_{t} v +v\cdot \nabla_{x,y} v + \nabla_{x,y} P 
= - g e_y ,\quad \text{div}\,_{x,y} v =0 \quad\text{in }\Omega,
\end{equation*}
where $g$ is the acceleration  due to gravity ($g>0$) and $P$ is the pressure. 
The problem is then given by three boundary conditions:
\begin{itemize}
\item a kinematic condition (which states that the free surface moves with the fluid) 
\begin{equation}
\partial_{t} \eta = \sqrt{1+|\partialx \eta|^2 } \, (v \cdot n)   \quad \text{ on } \Sigma,
\end{equation}
where $n$ is the unit  exterior normal  to $\Omega(t)$,
\item
a dynamic condition (that expresses a balance of forces across the free surface)  
\begin{equation}\label{syst}
P=0 \quad \text{ on } \Sigma, 
\end{equation}
\item  the "solid wall" boundary condition at the bottom $\Gamma$
\begin{equation}
v\cdot \nu=0, 
\end{equation}
\end{itemize}
where $\nu$ is the normal vector to $\Gamma$ 
whenever it exists. In the case of arbitrary bottom this condition will be implicit and contained in a variational formulation.

\subsubsection*{The Zakharov/Craig-Sulem formulation}

A popular form of the water-waves system 
is given by the Zakharov/Craig-Sulem 
formulation. This is an elegant 
formulation of the water-waves equations where 
all the unknowns are evaluated at the free surface only. 
Let us recall the derivation of this system. 

Assume, furthermore, that the motion of the liquid is irrotational. 
The velocity field $v$ 
is therefore given by $v=\nabla_{x,y} \Phi$ 
for some velocity potential $\Phi\colon \Omega\rightarrow \xR$ satisfying 
$$
\Delta_{x,y}\Phi=0\quad\text{in }\Omega,
\qquad \partial_\nu \Phi =0\quad\text{on }\Gamma,
$$
and the Bernoulli equation
\begin{equation}\label{Bern}
\partial_{t} \Phi +\mez \la \nabla_{x,y}\Phi\ra^2 + P +g y = 0 \quad\text{in }\Omega.
\end{equation}
Following Zakharov~\cite{Zakharov1968}, 
introduce the trace of the potential on the free surface:
$$
\psi(t,x)=\Phi(t,x,\eta(t,x)).
$$
Notice that since $\Phi$ is harmonic, $\eta$ and $\Psi$ fully determines $\Phi$. 
Craig and Sulem (see~\cite{CrSu}) observe that one can 
form a system of two evolution equations for $\eta$ and 
$\psi$. To do so, they introduce the Dirichlet-Neumann operator $G(\eta)$ 
that relates $\psi$ to the normal derivative 
$\partial_n\Phi$ of the potential by 
\begin{align*}
(G(\eta) \psi)  (t,x)&=\sqrt{1+|\partialx\eta|^2}\,
\partial _n \Phi\arrowvert_{y=\eta(t,x)}\\
&=(\partial_y \Phi)(t,x,\eta(t,x))-\nabla_x \eta (t,x)\cdot (\nabla_x \Phi)(t,x,\eta(t,x)). 
\end{align*}
(For the case with a rough bottom, we recall the precise construction later on). 
Directly from this definition, one has
\begin{equation}\label{eq:1}
\partial_t \eta=G(\eta)\psi.
\end{equation}
It is proved in~\cite{CrSu} (see also the computations in~\S\ref{S:pressure}) 
that the condition $P=0$ on the free surface implies that
\begin{equation}\label{eq:2}
\partial_{t}\psi+g \eta
+ \frac{1}{2}\la\partialx \psi\ra^2  -\frac{1}{2}
\frac{\bigl(\partialx  \eta\cdot\partialx \psi +G(\eta) \psi \bigr)^2}{1+|\partialx  \eta|^2}
= 0.
\end{equation}
The system~\eqref{eq:1}--\eqref{eq:2} is in Hamiltonian form 
(see~\cite{CrSu,Zakharov1968}), where the Hamiltonian is given by
$$
\mathcal{H}=\mez \int_{\xR^d} \psi G(\eta)\psi +g\eta^2\, dx.
$$

The problem to be considered here is that of the equivalence of the previous two formulations of the water-waves problem. 
Assume that the Zakharov/Craig-Sulem system has been solved. Namely, assume 
that, for some $r>1+ d/2$, 
$(\eta, \psi) \in C^0(I, H^{r}(\xR^d)\times  H^{r}(\xR^d))$ solves~\eqref{eq:1}-\eqref{eq:2}. 
We would like to show that we have indeed solved the initial system of Euler's equation with free boundary. In particular we have to define the pressure which does not appear in the above system~\eqref{eq:1}-\eqref{eq:2}. To do so we set 
$$ B=\frac{\partialx \eta \cdot\partialx \psi+ G(\eta)\psi}{1+|\partialx  \eta|^2},
\qquad
V= \partialx \psi -B \partialx\eta.$$
Then $B$ and $V$ belong to the space $C^0(I, H^{\mez}(\xR^d)).$
  It follows  from~\cite{ABZ1} that (for fixed $t$) one can define   unique variational solutions to the problems
$$
\Delta_{x,y}\Phi=0\quad\text{in }\Omega,\qquad 
\Phi\arrowvert_\Sigma =\Psi,\qquad \partial_\nu \Phi =0\quad\text{on }\Gamma.
$$
 $$
\Delta_{x,y} Q=0\quad\text{in }\Omega,\qquad 
 Q\arrowvert _\Sigma =g \eta + \mez(B^2 + \vert V \vert^2) ,\qquad \partial_\nu Q =0\quad\text{on }\Gamma.
$$

Then we shall define $P\in \mathcal{D}'(\Omega)$   by
$$
P:=  Q-gy-\mez \la \nabla_{x,y}\Phi\ra^2
$$
and we shall show firstly that $P$ has a trace on $\Sigma$ which is equal to $0$ and secondly  that $ Q =-\partial_t \Phi  $ which will show, according to \eqref{Bern} that we have indeed solved Bernouilli's (and therefore Euler's) equation. 

These assertions are not straightforward because we are working with  solutions of low regularity and we consider general bottoms (namely no regularity assumption is assumed on the bottom).  
Indeed, the analysis would have been much easier for $r>2+d/2$ and a flat bottom.

\bigbreak 
\noindent\textbf{Acknowledgements.} 
T.A. was supported by the French Agence Nationale de la Recherche, projects ANR-08-JCJC-0132-01 and ANR-08-JCJC-0124-01.

\section{Low regularity Cauchy theory}

Since we are interested in low regularity solutions, we begin by recalling the well-posedness results proved in~\cite{ABZ3}. These results clarify the Cauchy theory of 
the water waves equations as well in terms of regularity indexes for the initial conditions as for the smoothness of the bottom of the domain (namely no regularity assumption is assumed on the bottom).

Recall that the Zakharov/Craig-Sulem system reads
\begin{equation}\label{system}
\left\{
\begin{aligned}
&\partial_{t}\eta-G(\eta)\psi=0,\\[0.5ex]
&\partial_{t}\psi+g \eta
+ \frac{1}{2}\la\partialx \psi\ra^2  -\frac{1}{2}
\frac{\bigl(\partialx  \eta\cdot\partialx \psi +G(\eta) \psi \bigr)^2}{1+|\partialx  \eta|^2}
= 0.
\end{aligned}
\right.
\end{equation}

It is useful to introduce the vertical and horizontal components of the velocity,
\begin{gather*}
B:= (v_y)\arrowvert_{y=\eta} = (\partial_y \Phi)\arrowvert_{y=\eta},\quad 
V := (v_x)\arrowvert_{y=\eta}=(\nabla_x \Phi)\arrowvert_{y=\eta}.
\end{gather*}
These can be defined  in terms of $\eta$ and $\psi$ by means of the formulas 
  \begin{equation}\label{defi:BV}
B=\frac{\partialx \eta \cdot\partialx \psi+ G(\eta)\psi}{1+|\partialx  \eta|^2},
\qquad
V= \partialx \psi -B \partialx\eta.
\end{equation}
Also, recall that the Taylor coefficient $a=-\partial_y P\arrowvert_{\Sigma}$  
can be defined in terms of $\eta,V, B,\psi$ only (see~\S$4.3.1$ in~\cite{LannesLivre}).

In \cite{ABZ3} we  proved the following results about low regularity solutions. 
We refer to the introduction of~\cite{ABZ3,LannesJAMS} 
for references and a short historical survey of the background of this problem.
\begin{theo}[~\cite{ABZ3}]
\label{theo:Cauchy}
Let $d\ge 1$, $s>1+d/2$ and consider an initial data $(\eta_{0},\psi_{0})$ such that

$(i)$  
  $\eta_0\in H^{s+\mez}(\xR^d),\quad \psi_0\in H^{s+\mez}(\xR^d),\quad V_0\in H^{s}(\xR^d),\quad B_0\in H^{s}(\xR^d)$,
  
$(ii)$  the condition ($H_0$) in \eqref{eta*} holds initially for $t=0$,

$(iii)$  there exists a positive constant $c$ such that, for all $x$  in $\xR^d$, 
$a_0(x)\geq c$.

Then there exists $T>0$ such that 
the Cauchy problem for \eqref{system} 
with initial data  $(\eta_{0},\psi_{0})$ has a unique solution 
$$
(\eta,\psi)\in C^0\big([0,T], H^{s+\mez}(\xR^d)\times H^{s+\mez}(\xR^d)\big),
$$
such that 
\begin{enumerate}
\item   $(V,B)\in C^0\big([0,T], H^{s}(\xR^d)\times H^{s}(\xR^d)\big)$,
\item the condition ($H_t)$ in \eqref{eta*}   holds for $t\in [0,T]$ with $h$ replaced by 
$h/2$,
\item $a(t,x)\ge c/2,$    for all $(t,x)$ in $[0,T] \times \xR^d$. 
  \end{enumerate}
\end{theo}

\begin{theo}[~\cite{ABZ3}]\label{theo.strichartz}
Assume $\Gamma = \emptyset$. Let $d=2$, $s>1+\frac{d}{2}- \frac{1}{12}$ and consider an initial data 
 $(\eta_{0},\psi_{0})$ such that
\begin{equation*}
\eta_0\in H^{s+\mez}(\xR^d),\quad \psi_0\in H^{s+\mez}(\xR^d), \quad V_0\in H^{s}(\xR^d),
\quad B_0\in H^{s}(\xR^d).
\end{equation*}
Then there exists $T>0$ such that 
the Cauchy problem for \eqref{system} 
with initial data  $(\eta_{0},\psi_{0})$ has a  solution $(\eta, \psi)$ such that 
$$
(\eta,\psi,V,B)\in C^0\big([0,T];H^{s+\mez}(\xR^d)\times H^{s+\mez}(\xR^d) \times H^{s}(\xR^d)\times H^{s}(\xR^d)\big).
$$
\end{theo}
\begin{rema}
$(i)$ For the sake of simplicity we   stated   Theorem \ref{theo.strichartz} in dimension $d=2$ 
(recall that $d$ is the dimension of the interface). One can prove such a result in any dimension $d \geq 2,$ the number 
$1/12$ being replaced by an index depending on $d$.  

$(ii)$ Notice that in infinite depth ($\Gamma = \emptyset$) the Taylor  condition 
(which is assumption $(iii)$ in Theorem~\ref{theo:Cauchy}) is always satisfied as proved by Wu~(\cite{WuInvent}). 
\end{rema}
Now having solved the system \eqref{system} in $(\eta, \psi)$ we have to show that we have indeed solved the initial system in $(\eta,v)$.  This is the purpose of the following section.

There is one point that should be emphasized concerning the regularity. 
Below we consider solutions $(\eta,\psi)$ of \eqref{system} such that
$$
(\eta,\psi)\in C^0\big([0,T];H^{s+\mez}(\xR^d)\times H^{s+\mez}(\xR^d)),
$$
with the only assumption that $s>\mez +\frac{d}{2}$ 
(and the assumption that there exists 
$h>0$ such that the condition ($H_t)$ in \eqref{eta*} 
holds for $t\in [0,T]$). Consequently, the result proved 
in this note apply to the settings considered in the above theorems. 

\section{From Zakharov to Euler}

\subsection{The variational theory}

In this paragraph the time is fixed so we will skip it and work in a fixed domain $\Omega$ whose top boundary $\Sigma$ is Lipschitz i.e $\eta \in W^{1,\infty}(\xR^d).$
  
We recall here the variational theory, developed in~\cite{ABZ1}, 
allowing us to solve the following problem in the  case of arbitrary bottom,
\begin{equation}\label{dirichlet}
\Delta \Phi  = 0 \quad \text{in } \Omega, \quad  \Phi\arrowvert_{\Sigma}  = \psi, \quad  \frac{\partial \Phi}{\partial\nu}\arrowvert_\Gamma  = 0.
\end{equation}
Notice that $\Omega$ is not necessarily bounded below.  We proceed as follows.
 
Denote by $\mathcal{D}$ the space of functions $u\in C^\infty(\Omega)$ such that $\nabla_{x,y} u\in L^2(\Omega) $ and let $\mathcal{D}_0 $ be the subspace of functions $u \in \mathcal{D}$ such that $u$ vanishes near the top boundary~$\Sigma.$

\begin{lemm}[see Prop 2.2 in~\cite{ABZ1}]
There exist a positive weight $g\in L^\infty_{loc}(\Omega)$ 
equal to $1$ 
near the top boundary $\Sigma$ of $\Omega$ and 
$C>0$ such that for all $u\in \mathcal{D}_0$
\begin{equation}\label{poincare}
\iint_\Omega g(x,y) \vert u(x,y)\vert^2 dx dy 
\leq C \iint_\Omega \vert \nabla_{x,y}u(x,y) \vert^2 dx dy.
\end{equation}
\end{lemm}

Using this lemma one can prove the following result.

\begin{prop}[see page 422 in~\cite{ABZ1}]\label{hilbert}
Denote by $H^{1,0}(\Omega)$ the space of functions 
$u$ on $\Omega$ such that there exists a sequence 
$(u_n) \subset \mathcal{D}_0$ such that
$$\nabla_{x,y}u_n \to \nabla_{x,y} u \quad \text{ in } 
L^2(\Omega), \quad u_n \to u \quad \text{ in } L^2(\Omega, g  dx dy),
$$
endowed with the scalar product
$$
(u , v)_{H^{1,0}(\Omega)}  =  (\nabla_x u ,\nabla_x v)_{L^2(\Omega)} 
+  (\partial_y u ,\partial_yv)_{L^2(\Omega)}.
$$
Then  $H^{1,0}(\Omega)$ is a Hilbert space and \eqref{poincare} holds for $u \in H^{1,0}(\Omega).$
\end{prop}
 
Let $\psi \in H^{\mez}(\xR^d)$. 
One can construct  (see 
below after \eqref{u})   $\underline{\psi} \in H^1(\Omega)$ 
such that
$$
\supp \underline{\psi}\subset \{(x,y): \eta (t,x)-h\leq y \leq \eta(x)\}, 
\quad \underline{\psi}\arrowvert_{\Sigma} = \psi.
$$
Using Proposition~\ref{hilbert} we deduce that there exists a unique $u \in H^{1,0}(\Omega)$ 
such that, for all $\theta \in H^{1,0}(\Omega)$,
$$
\iint_\Omega \nabla_{x,y}u(x,y) \cdot \nabla_{x,y} \theta (x,y)dx dy 
= -\iint_\Omega \nabla_{x,y}\underline{\psi}(x,y) \cdot\nabla_{x,y}
 \theta (x,y)dx dy.
 $$

Then to solve the problem \eqref{dirichlet} we set $\Phi = u + \underline{\psi}.$ 

\begin{rema}
As for the usual Neumann problem  the meaning 
of the third condition in \eqref{dirichlet} is included 
in the definition of the space $H^{1,0}(\Omega). $ 
It can be written as in \eqref{dirichlet}  if the bottom $\Gamma$ is sufficiently smooth.  
\end{rema}
\subsection{The main result}
Let us assume that the Zakharov system  \eqref{system} 
has been solved on $I =(0,T)$, which means that we have found, for $ s>\mez + \frac{d}{2},$ a solution    
$$
(\eta, \psi) \in C^0(\overline{I}, H^{s+\mez}(\xR^d)\times  H^{s+\mez}(\xR^d)) ,
$$
 of the system

\begin{equation}\label{Zaharov}
\left\{
\begin{aligned}
&\partial_t \eta = G(\eta) \psi,\\
& \partial \psi = - g \eta - \mez \vert \nabla  \psi \vert ^2 
+\mez \frac{(\nabla \psi \cdot \nabla \eta + G(\eta) \psi)^2}{1 + \vert \nabla \eta \vert^2}.
\end{aligned}
\right.
 \end{equation}
Let $B,V$ be defined by \eqref{defi:BV}. Then $(B,V) \in C^0(I, H^{s-\mez}(\xR^d) \times H^{s-\mez}(\xR^d)).$ 
 

The above variational theory shows that one can solve (for fixed $t$) the problem
\begin{equation}\label{eq:Q}
\Delta_{x,y} Q = 0 \quad \text{in } \Omega, \quad Q\arrowvert_ \Sigma = g \eta + \mez(B^2 + \vert V \vert^2)\in H^{\mez}(\xR^d).
\end{equation}
     Here is the main result of this article.
\begin{theo}\label{maintheo}
 Let $\Phi$ and $Q$ be the variational solutions of the problems  \eqref{dirichlet} and  \eqref{eq:Q}. Set $P = Q-gy -\mez\vert \nabla_{x,y}\Phi\vert^2.$ Then $v := \nabla_{x,y}\Phi$ satisfies the Euler system
$$\partial_{t} v +(v\cdot \nabla_{x,y}) v + \nabla_{x,y} P 
= - g e_y   \quad \text{in }\Omega,  $$
together with the conditions    
\begin{equation}
\left\{
\begin{aligned}
 &\text{div}\,_{x,y} v =0,  \quad  \text{curl}\, _{x,y}  v =0  \quad \text{in } \Omega, \\ 
 & \partial_t \eta =  (1+| \nabla\eta|^2)^\mez \, (v\cdot n) \quad \text{on }  \Sigma, \\
& P= 0 \quad \text{on } \Sigma.
\end{aligned}
\right.
\end{equation}
 \end{theo}
The rest of the paper is devoted to the proof of this result. We proceed in several steps.

\subsection{Straightenning the free boundary}
   
First of all if condition $(H_t)$ is satisfied  on $I,$ for $T$ 
small enough, one can find $\eta_*\in L^\infty(\xR^d)$ independent of $t$ such that 
\begin{equation}\label{eta}
\left\{
\begin{aligned}
&(i)\quad  \nabla_x \eta_* \in H^\infty(\xR^d),\quad \Vert \nabla_x \eta_* \Vert_{L^\infty(\xR^d)}\leq C \Vert \eta \Vert_{L^\infty(I,H^{s+\mez}(\xR^d))},\\ 
&(ii) \quad \eta(t,x)-h \leq \eta_*(x) \leq \eta(t,x)- \frac{h}{2}, \quad \forall (t,x)\in I\times \xR^d,\\
&(iii) \quad \Gamma \subset \{(x,y) \in \mathcal{O}: y<\eta_*(x)\}.
\end{aligned}
\right.
\end{equation}
Indeed  using the first equation in \eqref{Zaharov} we have 
\begin{equation*}
\begin{aligned}
 \Vert \eta(t,\cdot)-\eta_0\Vert_{L^\infty(\xR^d)} &\leq \int_0^t\Vert G(\eta)\psi(\sigma,\cdot)\Vert_{H^{s-\mez}(\xR^d)} d\sigma\\ 
 &\leq TC\big(\Vert(\eta, \psi)\Vert_{L^\infty(I, H^{s+\mez}(\xR^d) \times H^{s+\mez}(\xR^d))}\big).
 \end{aligned}
 \end{equation*}
Therefore taking $T$ small enough 
we make  $\Vert \eta(t,\cdot)-\eta_0\Vert_{L^\infty(\xR^d)}$ 
as small as we want. Then we take      
$ \eta_*(x) = -\frac{2h}{3} + e^{-\delta\vert D_x \vert}\eta_0 $ 
and   writing
$$\eta_*(x) = -\frac{2h}{3} + \eta(t,x) - (\eta(t,x) - \eta_0( x)) 
+ (e^{-\delta\vert D_x \vert}\eta_0 - \eta_0(x)),
$$
we obtain \eqref{eta}.

In what follows we shall set 
\begin{equation}\label{lesomega}
\left\{
\begin{aligned}
&\Omega_1(t)  = \{(x,y): x \in \xR^d, \eta_*(x)<y< \eta(t,x)\},\\
   &\Omega_1 = \{(t,x,y): t\in I, (x,y)\in \Omega_1(t)\}, 
 \quad \Omega_2  = \{(x,y)\in \mathcal{O}: y\leq\eta_*(x)\},\\ 
&\tilde{\Omega}_1 = \{(x,z): x \in \xR^d, z \in (-1,0)\},\\  
 &\tilde{\Omega}_2 = \{(x,z) \in \xR^d \times (-\infty,-1]: (x, z+1+\eta_*(x))  \in \Omega_2\} \\
 &\tilde {\Omega} = \tilde{\Omega_1} \cup \tilde{\Omega_2} 
   \end{aligned}
\right.
 \end{equation}

Following Lannes~(\cite{LannesJAMS}), for $t\in I$ consider the map $(x,z)\mapsto (x, \rho(t,x,z))$ 
from $\tilde{\Omega} $ to $\xR^{d+1}$ defined by
\begin{equation}\label{diffeo}
\left\{
\begin{aligned}
\rho(t,x,z) &=  (1+z)e^{\delta z \langle D_x \rangle} \eta(t,x) -z \eta_*(x) \quad \text{if } (x,z) \in \tilde{\Omega}_1\\
\rho(t,x,z) &= z+1+ \eta_*(x) \quad  \text{if } (x,z) \in \tilde{\Omega}_2.
\end{aligned}
\right.
\end{equation}
where $\delta$ is chosen such that
$$\delta \Vert \eta \Vert_{L^ \infty(I, H^{s+\mez}(\xR^d))}:= \delta_0 <<1.$$
 
 Notice that since $s>\mez + \frac{d}{2},$ taking $\delta $ 
 small enough and using \eqref{eta} $(i), (ii),$ we obtain the estimates
\begin{equation}\label{rhokappa}
\begin{aligned}
(i)& \quad   \partial_z \rho(t ,x,z)   \geq \min\big({\frac{h}{3}, 1}\big)
\quad \forall (t,x,z)\in I \times \tilde{\Omega} ,\\
(ii)&\quad \Vert \nabla_{x,z} \rho\Vert_{ L^\infty(I \times \tilde{\Omega} )}
\leq C (1+\Vert \eta \Vert_{L^\infty(I,H^{s+\mez}(\xR^d))}) .\\
\end{aligned}
\end{equation}
It follows from \eqref{rhokappa} $(i)$  that the   map $(t,x,z) \mapsto (t,x,\rho(t,x,z))$ is a  diffeomorphism from $I \times \tilde{\Omega} $ to $\Omega$   which is of class $W^{1,\infty}.$   
 
We denote by $\kappa$ the inverse map of $\rho$:
\begin{equation}\label{kappa}
\begin{aligned}
(t,x,z)\in I\times \tilde{\Omega}  ,  (t,x,\rho(t,x,z)) 
&= (t,x,y)\\
  \Longleftrightarrow (t,x,z) &= (t,x,\kappa(t,x,y)),  (t,x,y) \in \Omega.
\end{aligned}
\end{equation}

\subsection{The Dirichlet-Neumann operator}
Let $\Phi$ be the variational solution described above (with fixed $t$) 
of the problem
\begin{equation}\label{var}
\left\{
\begin{aligned} 
&\Delta_{x,y} \Phi =0 \quad \text{in } \Omega (t),\\
&\Phi \arrowvert_{\Sigma(t)} = \psi(t,\cdot) ,\\
&\partial_\nu \Phi \arrowvert_{\Gamma} = 0.
\end{aligned}
\right.
\end{equation}
Let us recall that 
\begin{equation}\label{u}
\Phi = u + \underline{\psi}
\end{equation}
where $u \in H^{1,0}(\Omega(t))$ and $\underline{\psi}$ is an extension of $\psi$ to $\Omega (t).$ 
  
Here is a construction of $\underline{\psi.}$ 
Let  $\chi \in C^\infty(\xR), \chi(a)= 0$ if $a\leq-1,\chi(a) = 1$ if $a\geq -\mez.$ 
Let $\underline{\tilde{\psi}}(t,x,z) = \chi(z) e^{z \langle D_x \rangle} \psi(t,x)$ 
for $z\leq 0.$ It is classical that $ \underline{\tilde{\psi}}\in L^\infty(I,H^1(\tilde{\Omega}))$ 
if $\psi \in L^\infty(I,H^\mez(\xR^d))$ and
$$
\Vert \underline{\tilde{\psi}}  \Vert_{L^\infty(I, H^1(\tilde{\Omega}))}
\leq C \Vert \psi \Vert_{L^\infty(I, H^\mez(\xR^d))}.
$$
Then we set 
\begin{equation}\label{psisoul}
 \underline{\psi}(t,x,y)= \underline{\tilde{\psi}}\big(t, x, \kappa(t,x,y)\big).
 \end{equation}
Since $\eta \in C^0(I, W^{1,\infty}(\xR^d))$ 
we have $\underline{\psi}(t, \cdot)  \in H^1(\Omega(t)),\,  
\underline{\psi}\arrowvert_{\Sigma(t)} = \psi$ 
and 
$$
\Vert \underline{\psi}(t, \cdot) \Vert_{H^1(\Omega(t))} 
\leq C\big(\Vert \eta \Vert_{L^\infty(I, W^{1,\infty}(\xR^d))}\big)
\Vert \psi \Vert_{L^\infty(I, H^\mez(\xR^d))}.
$$
  
Then we define the Dirichlet-Neumann operator by
\begin{equation}
\begin{aligned}
G(\eta)\psi (t,x) &= \sqrt{1+\vert \nabla \eta \vert^2}\partial_n \Phi\arrowvert_{\Sigma}\\
&=(\partial_y \Phi)(x, \eta(t,x)) - \nabla_x \eta(t,x) \cdot(\nabla_x \Phi)(t,x,\eta(t,x)). 
\end{aligned}
\end{equation}
It has been shown in \cite{ABZ3} (see~$\S3$) that $G(\eta)\psi $ is well defined in 
$C^0(\overline{I},H^{-\mez}(\xR^d))$ 
if $\eta\in C^0(\overline{I},W^{1,\infty}(\xR^d))$ 
and $\psi \in C^0(\overline{I},H^{\mez}(\xR^d))$.
\begin{rema}
 Recall that we have set
\begin{equation}\label{omega}
\Omega(t) = \{(x,y) \in \mathcal{O}: y< \eta(t,x)\}, 
\quad \Omega   = \{(t,x,y): t\in I, (x,y)\in \Omega(t)\}.
\end{equation}
For a function $f\in L^1_{loc}(\Omega)$ if  $\partial_t f$ 
denotes its derivative in the sense of distributions we have
\begin{equation}\label{derf}
\langle \partial_t f, \varphi \rangle 
= \lim_{\eps \to 0}\Big\langle \frac{f(\cdot +\eps,\cdot,\cdot) - f(\cdot,\cdot,\cdot)}{\eps}, \varphi \Big\rangle, \quad \forall \varphi \in C_0^\infty(\Omega).
\end{equation}
This point should be clarified due to the 
particular form of the set $\Omega$ 
since we have to show that if $(t,x,y) \in \supp \varphi = K$ 
then 
$(t+\eps,x,y)\in \Omega $ for $\eps$ sufficiently small independently of the point $(t,x,y)$. 
This is true. Indeed if $(t,x,y) \in K$   there exists a fixed $\delta>0$ 
(depending only on $K,\eta $) such that $y\leq \eta(t,x) - \delta.$ 
Since by \eqref{Zaharov}
$$
\vert \eta(t +\eps,x) - \eta(t,x)\vert \leq \eps \Vert G(\eta)\psi \Vert_{L^\infty(I\times \xR^d)}\leq \eps C
$$ where 
$C=  C\big(\Vert(\eta,\psi)\Vert_{L^\infty(I,H^{s+\mez}(\xR^d)\times H^{s+\mez}(\xR^d))}\big),$ 
we have if $\eps < \frac{\delta}{C} $,
$$
y-\eta(t+\eps,x) = y-\eta(t ,x) +  \eta(t ,x) - \eta(t+\eps,x) \leq -\delta + \eps C<0.
$$

Notice that since 
$\eta \in C^0(\overline{I}, H^{s+\mez}(\xR^d)), \partial_t \eta = G(\eta) \psi \in C^0(\overline{I}, H^{s - \mez}(\xR^d))$ 
and \\ $s> \mez + \frac{d}{2}$ we have $\rho \in W^{1,\infty}(I\times \tilde{\Omega}).$
 \end{rema}
The main step in the proof of Theorem \ref{maintheo}  is the following.
\begin{prop}\label{theoprinc}
Let $\Phi$ be defined by \eqref{var} and $Q\in H^{1,0}(\Omega(t))$ by \eqref{eq:Q}. Then for all $t\in I$
\begin{equation*}
\begin{aligned}
& (i)  \quad \partial_t \Phi(t, \cdot) \in H^{1,0}(\Omega(t)),  \\
 & (ii)  \quad  \partial_t \Phi = -Q \quad \text{in } \mathcal{D}'(\Omega).
\end{aligned}
\end{equation*}
\end{prop}
This result will be proved in \S 3.6.
   \subsection{Preliminaries}
   
If $f $ is a function defined on $\Omega$ 
we shall denote by $\tilde{f}$ its image 
by the diffeomorphism $(t,x,z)\mapsto (t,x,\rho(t,x,z))$. 
Thus we have 
\begin{equation}\label{image}
\tilde{f}(t,x,z)  = f(t,x,\rho(t,x,z))     \Leftrightarrow f(t,x,y)  = \tilde{f}\big(t,x,\kappa(t,x,y)\big).  
\end{equation}
Formally  we have the following equalities for $(t,x,y) =  (t,x,\rho(t,x,z)) \in \Omega$ 
and $\nabla = \nabla_x$
\begin{equation}\label{dt}
\left\{
\begin{aligned}
\partial_yf(t,x,y) &= \frac{1}{\partial_z \rho}\partial_z \tilde{f}(t,x,z)
 \Leftrightarrow \partial_z \tilde{f}(t,x,z) = \partial_z \rho(t,x,\kappa(t,x,y)) \partial_y f(t,x,y),\\
\nabla  f(t,x,y) &= \big(\nabla \tilde{f} - \frac{\nabla  \rho}{\partial_z \rho}\, \partial_z\tilde{f}\big)(t,x,z) \Leftrightarrow  \nabla  \tilde{f}(t,x,z) = \big( \nabla f + \nabla \rho\,\,  \partial_y f\big)(t,x,y),\\
\partial_t  f(t,x,y) &= \big( \partial_t \tilde{f}  
+ \partial_t \kappa (t,x,y)  \partial_z\tilde{f}\big)(t,x,\kappa(t,x,y)). 
\end{aligned}
\right. 
\end{equation}
   
We shall set in what follows
\begin{equation}\label{lambda}
\Lambda_1 = \frac{1}{\partial_z \rho} \partial_z,\quad  \Lambda_2 
= \nabla_x - \frac{\nabla_x \rho}{\partial_z \rho}\partial_z 
\end{equation}
Eventually recall that if  $u$ is the function defined by \eqref{u} we have
\begin{equation}\label{egvar}
\iint_{\Omega(t)} \nabla_{x,y}u(t,x,y) \cdot \nabla_{x,y}\theta(x,y) dx dy 
= - \iint_{\Omega(t)} \nabla_{x,y}\underline{\psi}(t, x,y) \cdot \nabla_{x,y}\theta(x,y) dx dy 
\end{equation}
for all $\theta \in H^{1,0}(\Omega(t)) $ which implies  that for $t\in I,$
\begin{equation}\label{est}
 \Vert \nabla_{x,y} u(t, \cdot)\Vert_{L^2(\Omega(t))} \leq  C(\Vert \eta \Vert_{L^\infty(I, W^{1,\infty}(\xR^d)}) \Vert \psi \Vert_{L^\infty(I,H^\mez(\xR^d))}.
\end{equation}

Let $u$ be defined by \eqref{u}. 
Since $(\eta,\psi)\in C^0(\overline{I}, H^{s+\mez}(\xR^d)\times  H^{s+\mez}(\xR^d))$ 
the elliptic regularity theorem proved in \cite{ABZ3}, (see Theorem 3.16), 
shows that,  
$$
\partial_z \tilde{u}, \nabla_x \tilde{u} 
\in C_z^0([-1,0], H^{s-\mez}(\xR^d))\subset C^0([-1, 0] \times \xR^d),
$$
since $s-\mez > \frac{d}{2}.$

It follows from \eqref{dt} that $\partial_y u$ and $\nabla_x u$ have a trace on $\Sigma$ and 
$$
\partial_y u\arrowvert_\Sigma 
= \frac{1}{\partial_z \rho(t,x,0)} \partial_z \tilde{u}(t,x,0), \quad \nabla_x u\arrowvert_\Sigma 
= \big(\nabla_x \tilde{u} - \frac{\nabla_x \eta}{\partial_z \rho(t,x,0)} \partial_z \tilde{u}\big)(t,x,0).
$$
Since $\tilde{u}(t,x,0) =0$ it follows that
$$
\nabla_x u\arrowvert_\Sigma + (\nabla_x \eta) \partial_y u\arrowvert_\Sigma = 0
$$
from which we deduce, since $\Phi = u + \underline{\psi}$,  
\begin{equation}\label{debutV}
\nabla_x \Phi \arrowvert_\Sigma + (\nabla_x \eta) \partial_y \Phi\arrowvert_\Sigma = \nabla_x \psi.
\end{equation}
On the other hand one has
\begin{equation}\label{suiteV}
G(\eta)\psi= \big(\partial_y \Phi - \nabla_x \eta \cdot\nabla_x \Phi \big)\arrowvert_\Sigma.
\end{equation}
It follows from \eqref{debutV} and \eqref{suiteV} that we have 
\begin{equation}\label{finV}
\nabla_x \Phi\arrowvert_\Sigma = V, \quad \partial_y \Phi\arrowvert_\Sigma = B.
\end{equation}
According to \eqref{eq:Q}, $P = Q - gy - \mez\vert \nabla_{x,y} \Phi \vert^2$ has a trace on $\Sigma$ and $P\arrowvert_ \Sigma =0.$

\subsection{The regularity results}\label{S:pressure}

The main steps in the proof of Proposition \ref{theoprinc} are the following.
\begin{lemm}\label{dtphi}
Let  $\tilde{u}$ be defined by \eqref{image} and $\kappa$ by \eqref{kappa}. Then for all $t_0 \in I$
  the function $(x,y) \mapsto U(t_0,x,y):= \partial_t   \tilde{u}(t_0,x,\kappa (t_0,x,y))$ belongs to $H^{1,0}(\Omega(t_0)).$ Moreover there exists  a function $\mathcal{F}: \xR^+ \to \xR^+$ such that    $$ \sup_{t\in I} \iint_{\Omega(t)} \vert \nabla_{x,y}U(t,x,y)\vert^2 dx dy \leq \mathcal{F}(\Vert (\eta, \psi)\Vert_{L^\infty(I,H^{s+\mez}(\xR^d) \times H^{s+\mez}(\xR^d))}).$$
 \end{lemm}
 
\begin{lemm}\label{chain}
In  the sense of distributions on $\Omega $ we have the chain rule
$$\partial_t u(t,x,y) = \partial_t \tilde{u}(t,x,\kappa(t,x,y)) 
+ \partial_t \kappa(t,x,y) \partial_z\tilde{u}(t,x, \kappa(t,x,y)).
$$
\end{lemm}
These lemmas are proved in the next paragraph. 
\begin{proof}[of Proposition \ref{theoprinc}]
According to \eqref{u} and Lemma \ref{chain} we have
\begin{equation}\label{Dtphi}
 \partial_t \Phi(t,x,y) = \partial_t \tilde{u}(t,x,\kappa(t,x,y)) + \underline{w}(t,x,y)
 \end{equation}
where 
$$\underline{w}(t,x,y)  = \partial_t \kappa(t,x,y) \partial_z\tilde{u}(t,x, \kappa(t,x,y)) + \partial_t \underline{\psi}(t,x,y).$$

According to Lemma \ref{dtphi} the first term in the right hand side of \eqref{Dtphi} belongs to $H^{1,0}(\Omega(t)).$ Denoting by $\tilde{\underline{w}} $ the image of $\underline{w}$, if we show that  
\begin{equation}\label{underw}
\left\{
\begin{aligned}
&(i)\quad \tilde{\underline{w}} \in H^1(\xR^d \times \xR),\\
&(ii) \quad \supp{ \tilde{\underline{w}}} \subset \{(x,z)\in \xR^d \times (-1,0)\}\\
&(iii) \quad w\arrowvert_\Sigma = -g\eta - \mez(B^2 +\vert V \vert^2)
\end{aligned}\right.
\end{equation}
then $\partial_t \Phi$ will be the variational solution  of the problem
$$\Delta_{x,y} (\partial_t \Phi) = 0, \quad \partial_t \Phi \arrowvert_\Sigma =  -g\eta - \mez(B^2 +\vert V \vert^2).$$
By uniqueness, we deduce from  \eqref{eq:Q} that $\partial_t \Phi = -Q,$ which completes the proof of Proposition \ref{theoprinc}. Therefore we are left with the proof of \eqref{underw}.

Recall that $\tilde{\underline{\psi}}(t,x,z) = \chi(z) e^{z\langle D_x \rangle }\psi(t,x).$ Moreover by Lemma \ref{chain} we have 
$$ \widetilde{ \partial_t \underline{\psi}} =\Big( \partial_t \tilde{\underline{\psi}}  - \frac{\partial_t \rho(t,x,z)}{\partial_z\rho(t,x,z)} \partial_z\tilde{\underline{\psi}}\Big)(t,x,z). $$
Since $\psi \in   H^{s+\mez}(\xR^d), \partial_t \psi \in H^{s-\mez}(\xR^d),  \partial_t \eta \in H^{s-\mez}(\xR^d)$,  the classical properties of the Poisson kernel show that $\partial_t \tilde{\underline{\psi}} $ and $\partial_z\tilde{\underline{\psi}}$ and $\frac{\partial_t \rho }{\partial_z\rho } $ belong to $ H^s(\xR^d \times (-1,0))$ therefore to $H^1(\xR^d \times (-1,0)) $ since $s>\mez + \frac{d}{2}.$ It follows that the points $(i)$ and $(ii)$ in \eqref{underw} are satisfied by $\widetilde{ \partial_t \underline{\psi}}$. Now according to \eqref{diffeo} $\widetilde{\partial_t \kappa}$ is supported in $\xR^d \times (-1,0)$ and it follows from the elliptic regularity that $\widetilde{\partial_t \kappa}\partial_z \tilde{u}$ belongs to $H^1(\xR^d  \times (-1,0)).$ Let us check now point $(iii)$. Since $\partial_t \eta = G(\eta)\psi$ we have
\begin{equation}\label{trace}
  \partial_t \underline{\psi}(t,x,y)\arrowvert_\Sigma = \widetilde{ \partial_t \underline{\psi}} \arrowvert_{z=0} = \partial_t \psi - G(\eta)\psi  \cdot \partial_y \underline{\psi}(t,x,y) \arrowvert_\Sigma.
  \end{equation}
On the other hand we have
\begin{equation*}
\begin{aligned}
  \partial_t \kappa(t,x,y) \partial_z\tilde{u}(t,x, \kappa(t,x,y))\arrowvert_\Sigma &= \partial_t \kappa(t,x,y)\partial_z \rho(t,x, \kappa(t,x,y))   \partial_y {u}(t,x,y)\arrowvert_\Sigma\\
  &=  -\partial_t \rho(t,x, \kappa(t,x,y)) \partial_y {u}(t,x,y)\arrowvert_\Sigma\\
  &= -\partial_t \rho(t,x, \kappa(t,x,y)) \big(\partial_y\Phi(t,x,y) - \partial_y \underline{\psi}(t,x,y) \big)\arrowvert_\Sigma\\
  &= -G(\eta)\psi \cdot (B- \partial_y \underline{\psi}(t,x,y) \arrowvert_\Sigma.
  \end{aligned}
  \end{equation*}
  So using \eqref{trace} we find 
  $$ \underline{w}\, \arrowvert_\Sigma = \partial_t \psi - B G(\eta)\psi.$$
\end{proof}
 It follows from the second equation of \eqref{Zaharov} and from \eqref{defi:BV} that
 $$ \underline{w}\, \arrowvert_\Sigma =  -g\eta -\mez(B^2+ \vert V \vert^2).$$
This proves the claim $(iii)$ in \eqref{underw} and ends the proof of Proposition \ref{theoprinc}.

 \subsection{Proof  of the Lemmas}

\subsubsection{Proof of Lemma \ref{dtphi}}
 
Recall (see \eqref{diffeo} and \eqref{lambda}) that we have set 
\begin{equation*}\label{lambdaj}
\left\{
\begin{aligned}
\rho(t,x,z) &=  (1+z)e^{\delta z \langle D_x \rangle} \eta(t,x) -z \eta_*(x) \quad \text{if } (x,z) \in \tilde{\Omega}_1,\\
\rho(t,x,z) &= z+1+ \eta_*(x) \quad  \text{if } (x,z) \in \tilde{\Omega}_2,\\
\Lambda_1(t) &= \frac{1}{\partial_z \rho (t,\cdot)} \partial_z, \quad   \Lambda_2(t)  = \nabla_x - \frac{\nabla_x \rho(t, \cdot)}{\partial_z \rho(t,\cdot)}\partial_z 
\end{aligned}
\right.
 \end{equation*} 
 and that $\kappa_t$ has been defined in \eqref{kappa}. 
 
 If we set $\hat{\kappa_t}(x,y) = (x, \kappa(t, x,y))$ then $\hat{\kappa_t}$ is a bijective map from the space $H^{1,0}(\tilde{\Omega})$ (defined as in Proposition \ref{hilbert}) to the space $H^{1,0}( \Omega(t)). $ Indeed near the top boundary ($z\in (-2,0))$ this follows from the classical invariance of the usual space $H^1_0$ by a $W^{1,\infty}$-diffeomorphism,  while, near the bottom, our diffeomorphim is of class $H^\infty$ hence preserves the space $H^{1,0}.$

Now we fix  $t_0 \in I$, we take $\eps\in \xR\setminus{\{0\}}$ small enough and  we set for $t\in I$
\begin{equation}\label{FH}
\left\{
\begin{aligned}
F (t) &=  \iint_{\Omega (t)} \nabla_{x,y}u(t,x,y) \cdot \nabla_{x,y}\theta(x,y) dx dy \\
H  (t) &= - \iint_{\Omega (t)} \nabla_{x,y}\underline{\psi}(t, x,y) \cdot \nabla_{x,y}\theta(x,y) dx dy \\
 \end{aligned}
 \right.
 \end{equation}
 where $\theta  \in H^{1,0}(\Omega(t)) $ is chosen as follows.
 
 In  $F(t_{0}+\eps )$ we  take
  $$\theta_1 (x,y) = \frac{ u (t_0 + \eps,x,  y) - \tilde{u}(t_0, x,\kappa (t_0+ \eps, x,y))}{\eps} \in H^{1,0}(\Omega(t_0+ \eps)).$$
 
 In $F(t_0)$ we  take
 $$\theta_2 (x,y) = \frac{\tilde{u}(t_0 + \eps,x,  \kappa(t_0,x,y)) - u(t_0, x,y)}{\eps} \in H^{1,0}(\Omega(t_0)).$$
 Then in the variables $(x,z)$ we have
 \begin{equation}\label{thetaj}
 \begin{aligned}
& \tilde{\theta}_1(x,z) = \theta_1(x,\rho(t_0 +\eps,x,z)) = \frac{ \tilde{u} (t_0 + \eps,x, z) - \tilde{u}(t_0, x,z)}{\eps}\\
   & \tilde{\theta}_2(x,z) =  \theta_2(x,\rho(t_0,x,z)) =  \frac{ \tilde{u} (t_0 + \eps,x,  z) - \tilde{u}(t_0, x,z)}{\eps},
 \end{aligned}
\end{equation} 
 so we see that $\tilde{\theta}_1(x,z) =\tilde{\theta}_2(x,z)=:\tilde{\theta} (x,z)$.
 
It follows from \eqref{egvar} that for all $t\in I$ we have $F (t)   = H (t).$ Therefore
   $$
J _\eps(t_0) =: \frac{F (t_0 + \eps) - F (t_0)}{\eps}=\mathcal{J} _\eps(t_0) =: \frac{H (t_0 + \eps) - H (t_0)}{\eps}.
$$


Then after  changing  variables as in \eqref{diffeo} we obtain
\begin{equation}\label{J=sumK}
\begin{aligned}
J _\eps(t_0)  = \frac{1}{\eps}\sum_{j=1}^2 \iint_{\tilde{\Omega} }
\big[&\Lambda_j(t_0+\eps)\tilde{u}(t_0+\eps, x,z)\Lambda_j(t_0+\eps)
\tilde{\theta}(x,z)\partial_z \rho(t_0+ \eps,x,z) \\
-&\Lambda_j(t_0)\tilde{u}(t_0 , x,z)\Lambda_j(t_0)
\tilde{\theta}(x,z) \partial_z\rho(t_0,x,z)\big]dx dz 
=:\sum_{j=1}^2 K_{j,\eps}(t_0).
\end{aligned}
\end{equation}
 
With the notation used in \eqref{lambdaj} we can write,
\begin{equation}\label{champ}
\Lambda_j(t_0+ \eps) - \Lambda_j(t_0) = \beta_{j,\eps}(t_0,x,z) \partial_z, \quad j=1,2.
\end{equation}
Notice  that since the function $\rho$ does not depend on $t$ for $z \leq -1$ we have $ \beta_{j,\eps} =0$ in this set.

Then we have the following Lemma.
\begin{lemm}\label{estbeta}
There exists a non decreasing function $\mathcal{F}:\xR^+ \to \xR^+$ such that
$$
\sup_{t_0 \in I}\iint_{\tilde{\Omega} }\vert \beta_{j,\eps}(t_0,x,z)\vert^2 dxdz 
\leq \eps^2 \mathcal{F}\big(\Vert (\eta, \psi)\Vert_{L^\infty(I,H^{s+\mez}(\xR^d) \times H^{s+\mez}(\xR^d))}\big).
$$
\end{lemm}

\begin{proof}
 In the set $\{(x,z): x\in \xR^d, z \in (-1,0)\}$ the most delicate term to deal with  is
$$
(1) =:   \frac{\nabla_x \rho }{ \partial_z \rho} (t_0+ \eps,x,z) -   \frac{\nabla_x \rho }{ \partial_z \rho}(t_0,x,z)= \eps \int_0^1 \partial_t\Big( \frac{\nabla_x \rho}{\partial_z \rho} \Big)(t_0 + \eps \lambda,x,z)
d\lambda.
$$
We have
$$
\partial_t\Big( \frac{\nabla_x \rho}{\partial_z \rho} \Big) 
=\frac{\nabla_x \partial_t \rho}{\partial_z \rho}  
-\frac{(\partial_z \partial_t \rho)\nabla_x \rho}{(\partial_z \rho)^2}.
$$
First of all we have  $\partial_z \rho \geq \frac{h}{3}.$ 
Now since $s-\mez >\frac{d}{2} \geq \mez,$ 
we can write
\begin{equation}\label{nablarho}
\begin{aligned}
\Vert\nabla_x \partial_t \rho(t,\cdot) \Vert_{ L^2(\tilde{\Omega}_1) } 
&\leq 2\Vert e^{\delta z \vert D_x \vert}G(\eta)\psi(t,\cdot) \Vert_{L^2((-1,0),H^1(\xR^d))} \\
& \leq C\Vert G(\eta)\psi(t,\cdot)\Vert_{H^\mez (\xR^d)}\leq C\Vert G(\eta)\psi(t,\cdot)\Vert_{H^{s- \mez} (\xR^d)}\\
&\leq  C\big(\Vert (\eta, \psi)\Vert_{L^\infty(I,H^{s+\mez}(\xR^d) \times H^{s+\mez}(\xR^d))}\big).
\end{aligned}
\end{equation}
On the other hand we have
\begin{equation*}
\begin{aligned}
\Vert \nabla_x \rho(t,\cdot) \Vert_{L^\infty(\tilde{\Omega}_1)}
&\leq C\Vert e^{\delta z \vert D_x \vert}\nabla_x \eta(x,\cdot) \Vert_{L^\infty((-1,0), H^{s-\mez}(\xR^d)}
+ \Vert \nabla_x \eta_*\Vert_{L^\infty(\xR^d)}\\ 
&\leq C' \Vert \eta(t,\cdot) \Vert_{H^{s+\mez}(\xR^d)}
+\Vert \nabla_x \eta_*\Vert_{L^\infty(\xR^d)} \leq C"\Vert \eta(t,\cdot) \Vert_{H^{s+\mez}(\xR^d)}
\end{aligned}
\end{equation*}
 by \eqref{eta}. Eventually since 
$$
\partial_z \partial_t \rho 
= e^{\delta z \vert D_x \vert}G(\eta) \psi 
+ (1+z)\delta e^{\delta z \vert D_x \vert}\vert D_x \vert G(\eta) \psi,
$$
we have as in \eqref{nablarho}
\begin{equation}\label{estrho2}
\Vert\partial_z \partial_t \rho(t,\cdot) \Vert_{ L^2(\tilde{\Omega})}   \leq  C\big(\Vert (\eta, \psi)\Vert_{L^\infty(I,H^{s+\mez}(\xR^d) \times H^{s+\mez}(\xR^d))}\big).
\end{equation}

Then the Lemma follows.
\end{proof}

 Thus we can write for $j=1,2,$
\begin{equation}\label{J1=}
\begin{aligned}
K_{j,\eps}(t_0) &=  \sum_{k=1}^4\iint_{\tilde{\Omega}_1}A^k_{j,\eps}(t_0,x,z)dx dz, \quad  \\
 A^1_{j,\eps}(t_0,\cdot) 
 &= \Lambda_j(t_0)\Big[\frac {\tilde{u}(t_0+\eps,\cdot)-\tilde{u}(t_0,\cdot)}{\eps}\Big]
 \Lambda_j(t_0)\tilde{\theta}(\cdot)\partial_z\rho(t_0,\cdot),\\
A^2_{j,\eps}(t_0,\cdot)
&=\Big[\frac{\Lambda_j(t_0+\eps) -\Lambda_j(t_0)}{\eps}\Big]
\tilde{u}(t_0,\cdot)\Lambda_j(t_0)\tilde{\theta}(\cdot)\partial_z \rho(t_0,\cdot),\\
A^3_{j,\eps}(t_0,\cdot)
&=\Lambda_j(t_0+\eps)\tilde{u}(t_0+\eps,\cdot)
\Big[\frac{\Lambda_j(t_0+\eps) -\Lambda_j(t_0)}{\eps}\Big]\tilde{\theta}(\cdot)\partial_z\rho(t_0,\cdot),\\
A^4_{j,\eps}(t_0,\cdot)
&=\Lambda_j(t_0+\eps)\tilde{u}(t_0+\eps,\cdot)
\Lambda_j(t_0+\eps)\tilde{\theta}(\cdot)
\Big[\frac{\partial_z\rho(t_0+\eps, \cdot) - \partial_z\rho(t_0,\cdot)}{\eps}\Big].\\
\end{aligned}
\end{equation}
In what follows to simplify the notations we shall set $X=(x,z)\in \tilde{\Omega} $ and we recall that $ \Lambda_j(t_0+\eps) -\Lambda_j(t_0)=0$ when $z\leq -1.$

First of all, using the lower bound $\partial_z\rho(t_0,X) \geq \frac{h}{3}$, 
we obtain
\begin{equation}\label{A1}
\iint_{\tilde{\Omega} }A^1_{j,\eps}(t_0,X)dX \geq \frac{h}{3} \lA \Lambda_j(t_0)\Big[\frac{\tilde{u}(t_0 + \eps,\cdot) - \tilde{u}(t_0,\cdot)}{\eps} \Big] \rA^2_{L^2(\tilde{\Omega} )}.
\end{equation}

Now it follows from \eqref{champ} that
$$\la \iint_{\tilde{\Omega}}A^2_{j,\eps}(t_0,X)dX \ra 
\leq \sup_{t\in I}\Vert \frac{\beta_{j,\eps}}{\eps}\Vert_{L^2(\tilde{\Omega})}
\sup_{t\in I}\Vert \partial_z \tilde{u}(t,\cdot)\Vert_{L^\infty_z(-1,0, L^\infty(\xR^d))}\Vert
 \Lambda_j(t_0)\tilde{\theta}\Vert_{L^2(\tilde{\Omega} )}.
 $$
Since $s-\mez>\frac{d}{2}$  the elliptic regularity theorem shows that 
\begin{equation}\label{regtheo}
\begin{aligned}
\sup_{t\in I}\Vert \partial_z \tilde{u}(t,\cdot)\Vert_{L^\infty_z(-1,0, L^\infty(\xR^d))}&\leq \sup_{t\in I}\Vert \partial_z \tilde{u}(t,\cdot)\Vert_{L^\infty_z(-1,0, H^{s-\mez}(\xR^d))}\\
&\leq C\big(\Vert(\eta, \psi)\Vert_{L^\infty(I,H^{s+\mez}(\xR^d)\times H^{s+\mez}(\xR^d))}\big)
\end{aligned}
\end{equation}
Using Lemma \ref{estbeta} we deduce that
\begin{equation}\label{A2}
\la \iint_{\tilde{\Omega}}A^2_{j,\eps}(t_0,X)dX \ra 
\leq C\big(\Vert(\eta, \psi)\Vert_{L^\infty(I,H^{s+\mez}(\xR^d)
\times H^{s+\mez}(\xR^d))}\big)\Vert\Lambda_j(t_0)\tilde{\theta}\Vert_{L^2(\tilde{\Omega})}.
\end{equation}

Now write
\begin{multline*}
\iint_{\tilde{\Omega} }A^3_{j,\eps}(t_0,X)dX =\\
= \iint_{\tilde{\Omega} }\Lambda_j(t_0+ \eps)\tilde{u}(t_0+\eps,X)
\beta_\eps(t_0+\eps,X)\partial_z\tilde{\theta}(t_0,X)\partial_z\rho(t_0,X)dX.
\end{multline*}
By elliptic regularity, $\Lambda_j(t)\tilde{u}$ is bounded in $L^\infty_{t,x,z}$ by a fonction depending only on 
$$
\Vert(\eta, \psi)\Vert_{L^\infty(I,H^{s+\mez}(\xR^d)\times H^{s+\mez}(\xR^d))}.
$$  
Therefore we can write
\begin{equation}\label{A3}
\la \iint_{\tilde{\Omega}}A^3_{j,\eps}(t_0,X)dX \ra 
\leq C\big(\Vert(\eta, \psi)\Vert_{L^\infty(I,H^{s+\mez}(\xR^d)
\times H^{s+\mez}(\xR^d))}\big)\Vert\partial_z\tilde{\theta}\Vert_{L^2(\tilde{\Omega})}.
\end{equation}
Since
$$
\frac{\partial_z\rho(t_0+\eps, x,z) - \partial_z\rho(t_0,x,z)}{\eps} 
=\int_0^1  \partial_t \partial_z \rho(t_0+ \lambda\eps, x,z) d\lambda 
$$
(which vanishes when $z\leq  -1$), we find  using \eqref{regtheo} and \eqref{estrho2}
\begin{equation}\label{A4}
\la \iint_{\tilde{\Omega}}A^4_{j,\eps}(t_0,X)dX \ra \leq C\big(\Vert(\eta, \psi)\Vert_{L^\infty(I,H^{s+\mez}(\xR^d)\times H^{s+\mez}(\xR^d))}\big)\Vert\partial_z\tilde{\theta}\Vert_{L^2(\tilde{\Omega})}.
\end{equation}
 
  Now we consider
\begin{equation}\label{estdH}
\mathcal{J}_\eps = \frac{H (t_0+\eps)-H (t_0)}{\eps}. 
\end{equation}
We make the change of variable $(x,z)\to (x,\rho(t_0,x,z))$ 
in the integral and we decompose the new integral  as in \eqref{J=sumK}, 
\eqref{J1=}. This gives, with $X=(x,z)$,  
$$
\mathcal{J}_\eps   = \sum_{j=1}^2\mathcal{K}_{j,\eps}(t_0), 
\quad \mathcal{K}_{j,\eps}(t_0)
=    \sum_{k=1}^4\iint_{\tilde{\Omega}}\mathcal{A}^k_{j,\eps}(t_0,X)dX,
$$
where $\mathcal{A}^k_{j,\eps}$ has the same form as $-A^k_{j,\eps}$ in \eqref{J1=} 
except the fact that $\tilde{u}$ is replaced by $\underline{\tilde{\psi}}.$ 
Recall that $\underline{\tilde{\psi}}(t,x,z)  = \chi(z)e^{z\vert D_x\vert} \psi(t,x).$ Now we have
\begin{equation*}
\begin{aligned}
\Vert \Lambda_j \partial_t \underline{\tilde{\psi}}\Vert_{L^\infty(I, L^2(\tilde{\Omega}))} 
&\leq \mathcal{F}\big(\Vert  \eta\Vert_{L^\infty(I, H^{s+\mez}(\xR^d))}\big)\Vert 
\partial_t \underline{\tilde{\psi}}\Vert_{L^\infty(I, L_z^2((-1,0),H^1(\xR^d)))} \\
&\leq \mathcal{F}\big(\Vert  \eta\Vert_{L^\infty(I, H^{s+\mez}(\xR^d))}\big)
\Vert \partial_t\psi\Vert_{L^\infty(I, H^\mez (\xR^d))}\\
&\leq \mathcal{F}\big(\Vert  \eta\Vert_{L^\infty(I, H^{s+\mez}(\xR^d))}\big)
\Vert \partial_t\psi\Vert_{L^\infty(I, H^{s-\mez} (\xR^d))} 
\end{aligned}
\end{equation*}
since $s-\mez \geq \mez.$ 
Using the equation \eqref{Zaharov} on $\psi$, 
and the fact that $H^{s-\mez}(\xR^d)$ is an algebra we obtain
$$ \Vert \Lambda_j \partial_t \underline{\tilde{\psi}}\Vert_{L^\infty(I, L^2(\tilde{\Omega}))} \leq \mathcal{F}\big(\Vert  (\eta,\psi)\Vert_{L^\infty(I, H^{s+\mez}(\xR^d))\times  H^{s+\mez}(\xR^d))}  ).$$
It follows that we have
\begin{equation}\label{estAA1}
\la \iint_{\tilde{\Omega}}\mathcal{A}^1_{j,\eps}(t_0,X)dX\ra \leq  \mathcal{F}\big(\Vert(\eta, \psi)\Vert_{L^\infty(I,H^{s+\mez}(\xR^d)\times H^{s+\mez}(\xR^d))}\big) \Vert\Lambda_j(t_0)\tilde{\theta}\Vert_{L^2(\tilde{\Omega})}.
\end{equation}

Now since
\begin{equation*}
\begin{aligned}
\Vert \Lambda_j(t_0)\underline{\tilde{\psi}} (t,\cdot)\Vert_{L^2_z((-1,0),L^\infty(\xR^d))}
&\leq \mathcal{F}(\Vert \eta\Vert _{L^\infty(I, H^{s+\mez}(\xR^d))})\Vert
\underline{\tilde{\psi}}(t_0,\cdot)\Vert_{L^2_z((-1,0),H^{\frac{d}{2}+\eps}(\xR^d))}\\
&\leq \mathcal{F}(\Vert \eta\Vert _{L^\infty(I, H^{s+\mez}(\xR^d))})\Vert\psi\Vert_{L^\infty(I,H^{s+\mez}(\xR^d))}
\end{aligned}
\end{equation*}
we can use the same estimates as in \eqref{A2}, \eqref{A3}, \eqref{A4} to bound the terms $\mathcal{A}_{j,\eps}^k$ for $k = 2,3,4$. We obtain finally
  \begin{equation}\label{H}
\la \frac{H (t_0 + \eps) - H (t_0)}{\eps}\ra \leq C\big(\Vert(\eta, \psi)\Vert_{L^\infty(I,H^{s+\mez}(\xR^d)\times H^{s+\mez}(\xR^d))}\big)\sum_{j=1}^2 \Vert\Lambda_j(t_0)\tilde{\theta}\Vert_{L^2(\tilde{\Omega})}.
 \end{equation}
 Summing up using   \eqref{J1=}, \eqref{A1}, \eqref{A2},\eqref{A3},\eqref{A4}, \eqref{H} we find that, setting 
 $$ \tilde{U}_\eps(t_0,\cdot) = \frac{\tilde{u}(t_0 + \eps,\cdot) -  \tilde{u}(t_0,\cdot)}{\eps},$$
   there exists a non decreasing function $\mathcal{F}:\xR^+ \to \xR^+$ such that  for all $\eps>0 $  
 $$\sum_{j=1}^2 \sup_{t_0 \in I} \Vert \Lambda_j(t_0) \tilde{U}_\eps(t_0,\cdot)\Vert_{L^2(\tilde{\Omega})} \leq \mathcal{F} \big(  \Vert(\eta, \psi)\Vert_{L^\infty(I,H^{s+\mez}(\xR^d)\times H^{s+\mez}(\xR^d))}\big).$$
   Since $\tilde{u}(t,\cdot) \in H^{1,0}(\tilde{\Omega}), $    the Poincar\' e inequality ensures that
 \begin{equation}\label{Uborne}
 \lA\tilde{U}_\eps(t_0,\cdot)\rA_{L^2(\tilde{\Omega} )} \leq   C\big(\Vert(\eta, \psi)\Vert_{L^\infty(I,H^{s+\mez}(\xR^d)\times H^{s+\mez}(\xR^d))}\big). 
  \end{equation}
  
It follows that we can extract  a subsequence  $(\tilde{U}_{\eps_k})$ which converges in the weak-star topology of $(L^\infty \cap C^0)(I, H^{1,0}(\tilde{\Omega}))$.
  But this sequences converge  in $\mathcal{D}'(I\times \tilde{\Omega})$ 
  to $ \partial_t \tilde{u}. $ Therefore $\partial_t \tilde{u} \in C^0(I, H^{1,0}(\tilde{\Omega}))$ and this implies that $\partial_t \tilde{u}(t_0, \cdot, \kappa(t_0, \cdot,\cdot))$ belongs to  $H^{1,0}(\Omega(t_0) $ which completes the proof of Lemma \ref{dtphi}.  

\subsubsection{Proof of Lemma \ref{chain}}

Let $ \varphi \in C_0^\infty(\Omega)$ and set
\begin{equation}\label{I+J}
\begin{aligned}
&v_\eps(t,x,y)  = \frac{1}{\eps}[\tilde{u}(t+\eps,x,\kappa(t+\eps,x,y)) - \tilde{u}(t +\eps ,x,\kappa(t,x,y))], \\
&w_\eps(t,x,y)  = \frac{1}{\eps}[\tilde{u}(t+\eps,x,\kappa(t,x,y)) - \tilde{u}(t  ,x,\kappa(t,x,y))],\\
&J_\eps  = \iint_{\Omega}v_\eps(t,x,y) \varphi(t,x,y) dt dx dy, \quad K_\eps  = \iint_{\Omega}w_\eps(t,x,y) \varphi(t,x,y) dt dx dy,\\
&I_\eps = J_\eps + K_\eps.
\end{aligned}
\end{equation}
Let us consider first $K_\eps.$ In the integral in $y$ we make the change of variable   $\kappa(t,x,y) = z \Leftrightarrow y = \rho(t,x,z).$ Then setting  $\tilde{\varphi}(t,x,z) = \varphi(t,x,\rho(t,x,z))$ and $X = (x,z) \in \tilde{\Omega}$ we obtain
$$K_\eps = \iint_I \int_{\tilde{\Omega}}\frac{ \tilde{u}(t+\eps,X) - \tilde{u}(t ,X)}{\eps} \tilde{\varphi}(t,X) \partial_z \rho(t,X)dt dX.$$
Since $\rho \in C^1(I \times \tilde{\Omega})$  we have $\tilde{\varphi} \cdot \partial_z \rho \in C_0^0(I\times \tilde{\Omega}).$  Now we know that the sequence $\tilde{U}_\eps =  \frac{ \tilde{u}(\cdot+\eps,\cdot) - \tilde{u}(\cdot,\cdot)}{\eps} $ converges in $\mathcal{D}'(I\times\tilde{\Omega})$ to $ \partial_t\tilde{u}.$ We use  this fact, we  approximate  $\tilde{\varphi} \cdot \partial_z \rho$ by a sequence in $C_0^\infty(I\times\tilde{\Omega})$ and we use  \eqref{Uborne}     to deduce that

$$
\lim_{\eps \to 0} K_\eps = \int_I\iint_{\tilde{\Omega}} \partial_t \tilde{u}(t,X)\tilde{\varphi}(t,X) \partial_z \rho(t,X)dt dX.$$
Coming back to the $(t,x,y)$ variables we obtain   
 \begin{equation}\label{K}
 \lim_{\eps \to 0} K_\eps =  \iint_{ \Omega} \partial_t \tilde{u}(t,x,\kappa(t,x,y)) \varphi (t,x,y)  dt dx dy. 
  \end{equation}
Let us look now to $J_\eps.$ We cut it into two integrals; in the first we set $\kappa(t+\eps,x,y) =z$ in the second we set  $\kappa(t ,x,y) =z.$ With $X=(x,z)\in \tilde{\Omega}$ we obtain 
$$J_\eps = \frac{1}{\eps}\int_I \iint_{\tilde{\Omega}_1 }\tilde{u}(t + \eps,X) \Big( \int_0^1 \frac{d}{d\sigma}\big\{\varphi(t,x,\rho(t+\eps \sigma,X))\partial_z\rho(t+\eps \sigma,X)\big\}d\sigma\Big) dtdX.$$
Differentiating with respect to $\sigma$ we see easily that
$$J_\eps =   \int_I \iint_{\tilde{\Omega}}\tilde{u}(t+\eps,X) \frac{\partial}{\partial z}\Big( \int_0^1 \partial_t\rho(t+\eps \sigma,X) \varphi(t,x,\rho(t+\eps \sigma,X))d\sigma\Big) dt dX.$$
Since $\tilde{u}$ is continuous in $t$ with values in $L^2(\tilde{\Omega}),  \partial_t\rho $ is continous in $(t,x,z)$ and $\varphi \in C_0^\infty$ we can pass to the limit and we obtain
$$\lim_{\eps \to 0}J_\eps =   \int_I \iint_{\tilde{\Omega}_1}\tilde{u}(t ,X) \frac{\partial}{\partial z}\Big( \partial_t \rho(t,X) \varphi(t,x,\rho(t ,X)) \Big) dt dX.$$
Now we can integrate by parts. Since, thanks to $\varphi,$ we have compact support in $z$ we obtain
$$\lim_{\eps \to 0}J_\eps =  - \int_I \iint_{\tilde{\Omega}}\partial_z \tilde{u}(t ,X)\partial_t \rho(t,X) \varphi(t,x,\rho(t ,X))  dt dX.$$
Now since 
$$\partial_t \rho(t,X) = -\partial_t \kappa(t,x,y)\partial_z \rho(t,x,z)$$
 setting in the integral in $z$, $\rho(t,X) = y$ we obtain   
 \begin{equation}\label{Jeps}
 \lim_{\eps \to 0}J_\eps  = \iint_{\Omega}\partial_z \tilde{u}(t ,x,\kappa(t,x,y)) \partial_t\kappa(t,x,y) \varphi(t,x,y)  dt dx dy.
 \end{equation}
Then Lemma \ref{chain} follows from \eqref{I+J}, \eqref{K} and \eqref{Jeps}.

\end{document}